\documentclass[12pt]{article}
\usepackage{amsmath,amsfonts,amsthm,amscd,upref,amstext}
\usepackage[dvips]{graphicx}
\usepackage{subfigure}
\usepackage{pb-diagram,pb-xy}
\usepackage[all]{xy}

\newtheorem{prop}{Proposition}[section]
\newtheorem{thm}[prop]{Theorem}
\newtheorem{cor}[prop]{Corollary}

\newtheorem{defn}[prop]{Definition}
\newtheorem{rem}[prop]{Remark}
\newtheorem{lem}[prop]{Lemma}

\newcommand{\N}{\mathbb{N}}

\numberwithin{equation}{section}

\begin{document}

\title{(FC)-Sequences, Mixed Multiplicities and Reductions of Modules}
\author{R. Callejas-Bedregal$^{1,}\,$\thanks{Partially supported by CAPES-Brazil Grant Procad-190/2007, CNPq-Brazil Grant 620108/2008-8 and by
FAPESP-Brazil Grant 2010/03525-9.  2000 Mathematics Subject
Classification: 13H15(primary). {\it Key words}: Mixed
multiplicities, (FC)-sequences, Buchsbaum-Rim multiplicity,
reduction. }\,\,\,\,and\,\,\,V. H. Jorge P\'erez$^{2}$
\thanks{Work partially supported by CNPq-Brazil - Grant
309033/2009-8, Procad-190/2007, FAPESP Grant 09/53664-8.}}

\date{}
\maketitle

\noindent $^1$ Universidade Federal da Para\'\i ba-DM, 58.051-900,
Jo\~ao Pessoa, PB, Brazil ({\it e-mail: roberto@mat.ufpb.br}).

\vspace{0.3cm}
\noindent $^2$ Universidade de S{\~a}o Paulo -
ICMC, Caixa Postal 668, 13560-970, S{\~a}o Carlos-SP, Brazil ({\it
e-mail: vhjperez@icmc.usp.br}).

\vspace{0.3cm}
\begin{abstract}

Let $(R, \mathfrak m)$ be a Noetherian local ring. In this work we
use the notion of (FC)-sequences, as defined in
\cite{perez-bedregal1}, to present some results concerning
reductions and the positivity of mixed multiplicities of a finite
collection of arbitrary submodules of $R^p.$ We also investigate
the length of maximal (FC)-sequences. We actually work in the more
general context of standard graded $R$-algebras.
\end{abstract}

\maketitle
\section{Introduction}

The notion of mixed multiplicities for a family $E_1,\ldots, E_q$
of $R$-submodules of $R^p$ of finite colength, where $R$ is a
local Noetherian ring, have been described in a purely algebraic
form by Kirby and Rees in \cite{Kirby-Rees1} and in an
algebro-geometric form by Kleiman and Thorup in
\cite{Kleiman-Thorup} and \cite{Kleiman-Thorup2}. The results of
Risler and Teissier in \cite{Teissier} and of  Rees in \cite{Rees}
where generalized for modules in \cite{Kirby-Rees1} and
\cite{Bedregal-Perez}, where the mixed multiplicities for
$E_1,\ldots, E_q$ are described as the Buchsbaum-Rim multiplicity
of a module generated by a suitable joint reduction of
$E_1,\ldots, E_q.$

The question which arises as to what happens with the positivity
of the mixed multiplicities of arbitrary ideals and modules. In
order to answer this question in the case of ideals, Vi\^et in
\cite{Viet8} (see also \cite{Manh-Viet}) built a sequence of
elements, called an (FC)-sequence, and proved that mixed
multiplicities  of a set of arbitrary ideals could be described as
the Hilbert-Samuel multiplicity of the ideal generated by a
suitable (FC)-sequence. Similar descriptions were also obtained by
Trung in \cite{Trung1} using  the stronger notion of
filter-regular sequences. The notion of (FC)-sequences was
generalized by the authors in \cite{perez-bedregal1} for a family
of arbitrary modules were they proved that its mixed
multiplicities could be described as the Buchsbaum-Rim
multiplicity of the module generated by a suitable (FC)-sequence,
thus extending the main results of Vi\^et and Vi\^et and Manh
({\it loc. cit.}) to this new setting. The above mentioned results
show that (FC)-sequences hold important information on mixed
multiplicities.

In this work we present some results concerning reductions and the
vanishing and non-vanishing of  mixed multiplicities of a family
of arbitrary $R$-submodules $F, E_1,\ldots, E_q$ of  $R^p$ with
$F$ of finite colength in $R^p.$ We prove many new and more
general results than in Trung \cite{Trung1}, Vi\^et
(\cite{Viet10}, \cite{Viet8}) and  the authors in
\cite{perez-bedregal1}. In fact, we do this in the context of
standard graded algebras.

Fix a graded $R$-algebra $G=\oplus G_n,$ that, as usual, is
generated as algebra by finitely many elements of degree one and
$M$ a finitely generated graded $G$-module. This paper is divided
into six sections.

In Section 2, we recall the concept of Buchsbaum-Rim
multiplicities of $R$-submodules of $G_1$ associated to $M,$
introduced by Buchsbaum and Rim in \cite{Buchsbaum-Rim} for
modules and carried out in this generality by Kleiman and Thorup
in \cite{Kleiman-Thorup} and by Kirby and Rees in
\cite{Kirby-Rees1}.

In Section 3, we recall the notion of (FC)-sequences and
weak-(FC)-sequences of $R$-submodules of $G_1$ associated to $M,$
introduced in this context by the authors in
\cite{perez-bedregal1} and in the ideal case by Vi\^et in
\cite{Viet8}.

In Section 4, we recall the notion of mixed multiplicities of
arbitrary $R$-submodules of $G_1,$ as introduced by the authors in
\cite{perez-bedregal1}, and state some of its main properties
proved by the authors ({\it loc. cit.}).

In Section 5, we give some characterizations for the length of
maximal weak-(FC)-sequences and the relation between maximal
weak-(FC)-sequences and reductions of $R$-submodules of $G_1$ with
respect to $M$.

In Section 6, we describe how to apply to arbitrary modules the
results obtained in the previous sections.

\section{Buchsbaum-Rim multiplicities}

\vspace{0.3cm}  Fix $(R, \mathfrak m)$ an arbitrary Noetherian
local ring; fix a graded $R$-algebra $G=\oplus G_n,$ which as
usual is generated as algebra by finitely many elements of degree
one; fix $I$ a finitely generated $R$-submodule of $G_1$ such that
$\ell(G_1/I)<\infty;$ and fix $M$ a finitely generated graded
$G$-module. Let $r:=\dim(\mbox{Proj}(G))$ be the dimension of
$\mbox{Proj}(G).$ As a function of $n,q$, the length,
$$h(n,q):=\ell(M_{n+q}/I^nM_q)$$
\noindent is eventually a polynomial in $n,q$ of total degree
equal to $\dim(\mbox{Supp}(M)),$ which is at most $r,$ (see
\cite[Theorem 5.7]{Kleiman-Thorup}) and the coefficient of
$n^{r-j}q^j/(r-j)!j!$ is denoted by $e^j(I,M),$ for all
$j=0,\ldots, r,$  and it is called the {\it $j^{\mbox{th}}$
Associated Buchsbaum-Rim multiplicity of} $I$ {\it with respect
to} $M.$ Notice that $e^j(I,M)=0$ if $\dim(\mbox{Supp}(M))<r.$ The
number $e^0(I,M)$ will be called the {\it Buchsbaum-Rim
multiplicity of} $I$ {\it with respect to} $M,$ and will also be
denoted by $e_{BR}(I,M).$ The notion of Buchsbaum-Rim multiplicity
for modules goes back to \cite{Buchsbaum-Rim} and it was carried
out in the above generality in \cite{Kirby}, \cite{Kirby-Rees1},
\cite{Kleiman-Thorup}, \cite{Kirby-Rees2}, \cite{Kleiman-Thorup2}
and \cite{Simis-Ulrich-Vasconcelos}.

\section{FC-sequence}

{\bf Setup (1):} Fix $(R, \mathfrak m)$ an arbitrary Noetherian
local ring; fix a graded $R$-algebra  $G=\oplus G_n,$ that, as
usual, is generated as algebra by finitely many elements of degree
one; fix $J$ a finitely generated $R$-submodule of $G_1$ such that
$\ell(G_1/J)<\infty;$ fix $I_1,\dots, I_q$ with $I_i\subseteq
G_{1}$ finitely generated $R$-submodules; and fix $M=\oplus M_n$ a
finitely generated graded $G$-module generated in degree zero,
that is, $M_n=G_nM_0$ for all $n\geq 0.$ We denote by ${\mathcal
I}$ the ideal of $G$ generated by $I_1\cdots I_q$. Set
$M^*:=M/0_M:{\mathcal I}^{\infty}.$

We use the following multi-index notation through the remaining of
this work. The norm of a multi-index ${\bf r}=(r_1,\ldots, r_q)$
is $|{\bf r}|=r_1+\cdots+r_q$ and ${\bf r}!=r_1!\cdots r_q!.$ If
${\bf r,s}$ are two multi-index then ${\bf r}^{\bf
s}=r_1^{s_1}\cdots r_q^{s_q}.$  If ${\bf I}=(I_1,\ldots,I_k)$ is a
$k$-tuple of $R$-submodules of $G_1$ then ${\bf I}^{\bf n}:=
I_1^{n_1}\cdots I_k^{n_k}.$  We also use the following notation,
$\delta(i)=(\delta(i,1),\ldots, \delta(i,k)),$ where
$\delta(i,j)=1$ if $i=j$ and $0$ otherwise.

\begin{defn}\label{FC}
 Let $I_1,\dots,I_q$ be $R$-submodules of $G_1$ such that
$\mathcal I$ is not contained in $\sqrt{\mbox{Ann}M},$ where
$\mathcal I$ is the ideal in $G$ generated by $I_1\cdots I_q.$ We
say that an element $x\in G_1$ is an (FC)-element with respect to
$(I_1,...,I_q; M)$ if there exists
 an $R$-submodule $I_i$ of $G_1$ and an integer $r'_i$ such that
\begin{enumerate}
\item [$(FC_1)$] $x\in I_i\setminus {\mathfrak m}I_i$  and
$${\bf I}^{\bf r}M_p\cap xM_{|{\bold r}|+p-1}=x{\bf I}^{{\bf r}-\delta(i)}M_p$$ \noindent for all ${\bold r}\in \N^q$ with
$r_i\geq r'_i.$

\item [$(FC_2)$] $x$ is a filter-regular element with respect to $({\cal I}; M),$ i.e.,  \break $0_{M}:x\subseteq 0_{M}:{\mathcal I}^{\infty}.$

\item [$(FC_3)$] $\dim(\mbox{Supp}(M/xM:{\mathcal I}^{\infty}))=\dim(\mbox{Supp}(M^*))-1.$
\end{enumerate}

 \noindent We call  $x\in G_1$ a weak-(FC)-element with respect to $(I_1,...,I_q; M)$ if $x$ satisfies the conditions $(FC_1)$ and $(FC_2).$

 A sequence of elements $x_1,\ldots, x_k$ of $G_1$,  is said to be an (FC)-sequence  \break with
respect to $(I_1,...,I_q; M)$ if ${\overline x}_{i+1}$ is an
(FC)-element with respect  \break to
$({\overline{I}}_1,...,{\overline{I}}_q; \overline{M})$ for each
$i=1,...,q-1$, where $\overline{M}=M/(x_1,...,x_i)M$,  ${\overline
x}_{i+1}$ is the initial form of $x_{i+1}$ in
$\overline{G}=G/(x_1,\dots,x_i)$  and ${\overline
I}_i=I_i\overline{G}$, $i=1,\dots, q.$

A sequence of elements $x_1,\ldots, x_k$ of $G_1$,  is said to be
a weak-(FC)-sequence with respect to $(I_1,...,I_q; M)$ if
${\overline x}_{i+1}$ is a weak-(FC)-element with respect to
$({\overline{I}}_1,...,{\overline{I}}_q; \overline{M})$ for each
$i=1,...,q-1$.
\end{defn}

The following proposition which was proved in \cite[Proposition
2.3]{perez-bedregal1}, will show the existence of
weak-(FC)-sequences.

\begin{prop}\label{Obs1}
If ${\mathcal I}$  is not contained in $\sqrt{\mbox{Ann}M}$ then,
for any $i=1,\ldots,q,$ there exists a weak-(FC)-element $x_i\in
I_i$ with respect to $(I_1,...,I_q;M).$
\end{prop}

\begin{rem}\label{remark0}
{\rm From Proposition \ref{Obs1} it follows that a
weak-(FC)-sequence $x_1,...,x_p$ in $I_1\cup\cdots \cup I_q$ with
respect to $(I_1,...,I_q; M)$ is a maximal weak-(FC)-sequence if
and only if ${\mathcal I}\subseteq
\sqrt{\mbox{Ann}(M/(x_1,...,x_p)M)}.$ }
\end{rem}
 Hence, if ${\mathcal I}$  is not contained in $\sqrt{\mbox{Ann}M}$ then
there always exists a maximal weak-(FC)-sequence in $I_1\cup\cdots
\cup I_q$ with respect to $(I_1,...,I_q; M).$

\begin{lem}\label{lema0}
Let $I_1,...,I_q$ be $R$-submodules of $G_1$ such that the ideal
${\mathcal I}$ of $G$ generated by $I_1\cdots I_q$ is not
contained in $\sqrt{\mbox{Ann}M}.$ Let $J_1,...,J_t$ be
$R$-submodules of $G_1$ of finite colength. Let $x\in I_i$ be an
(FC)-element with respect to $(J_1,...,J_t,I_1,...,I_q;M).$ Then
$x$ is an (FC)-element with respect to $(I_1,...,I_q;M)$ and for
any $1\leq s\leq t,$ $x$ is also an (FC)-element with respect to
$(J_1,...,J_s,I_1,...,I_q;M).$
\end{lem}

\begin{proof}
Let ${\cal J}_s$ be the ideal of $G$ generated by $J_s,\, 1\leq
s\leq t.$ Since $J_1,...,J_t$ have finite colength in $G_1$ we
have that for all $s=1,\dots,t,$
$$(0_M:({\cal J}_1\cdots {\cal J}_t\cdot {\cal I})^{\infty})=(0_M:({\cal J}_1\cdots {\cal J}_s\cdot {\cal I})^{\infty})=
(0_M:{\cal I}^{\infty})$$ and
$$(xM:({\cal J}_1\cdots {\cal J}_t\cdot {\cal I})^{\infty})=(xM:({\cal J}_1\cdots {\cal J}_s\cdot {\cal I})^{\infty})=
(xM:{\cal I}^{\infty}).$$ But, since $x\in I_i$ is an (FC)-element
with respect to $(J_1,...,J_t,I_1,...,I_q;M)$ we have that
$$\dim\left ( \frac{M}{(xM:({\cal J}_1\cdots {\cal J}_t\cdot {\cal I})^{\infty})}\right )=
\dim\left ( \frac{M}{(0_M:({\cal J}_1\cdots {\cal J}_t\cdot {\cal
I})^{\infty})}\right )-1.$$ Hence the result follows.
\end{proof}

\section{Mixed multiplicities}

We keep the notations of  setup (1). In this section we recall the
notion of mixed multiplicities of $J,I_1,\dots, I_q$ with respect
to $M,$ as introduced by the authors in \cite{perez-bedregal1}.
For the reader convenience we state without proof some important
results of \cite{perez-bedregal1}. The main result of this section
establish mixed multiplicity formulas by means of Buchsbaum-Rim
multiplicities and also determines the positivity of mixed
multiplicities.

Consider the function
$$h(n,p,{\bold r}):=\ell \left(\frac{{\bf I}^{\bf r}M_{n+p}}
{J^{n}{\bf I}^{\bf r}M_{p}}\right ).$$ By \cite[eq.
(4.1)]{perez-bedregal1}, for all large $(n,p,\bold r)\in
\N^{q+2},$ we have that
\begin{equation}\label{equation0}
h(n,p,\bold r)=\ell \left(\frac{{\bf I}^{\bf r}M_{n+p}} {J^{n}{\bf
I}^{\bf r}M_{p}}\right )=\ell \left(\frac{{\bf I}^{\bf
r}M^*_{n+p}} {J^{n}{\bf I}^{\bf r}M^*_{p}}\right )\end{equation}
which by \cite[Theorem 4.1]{perez-bedregal1} is a polynomial of
degree $D:=\dim(\mbox{Supp}(M^{*})).$   If we write the terms of
total degree $D$ of the polynomial $h(n,p,{\bold r})$ in the form
$$B(n,p,\bold r)=\sum_{k_0+|{\bf k}|+j=D}\;\frac{1}{k_0!{\bf k}!j!}e^j(J^{[k_0]},I_1^{[k_1]},\dots,I_q^{[k_q]};M)
{\bf r}^{\bf k}n^{k_0}p^j.$$ \noindent The coefficients
$e^j(J^{[k_0]},I_1^{[k_1]},\dots,I_q^{[k_q]};M)$ are called the
$j^{\mbox {th}}$-{\it mixed multiplicities} of $(J, I_1,\ldots,
I_q; M)$. We call $e^0(J^{[k_0]},I_1^{[k_1]},\dots,I_q^{[k_q]};M)$
the {\it mixed multiplicity} of $(J, I_1,\ldots, I_q; M)$ of type
$(k_0,k_1,\dots,k_q)$.

\begin{thm}\label{Teo1.2}{\rm \cite[Theorem 4.6]{perez-bedregal1}}.
Keeping the setup $(1),$ assume that $D>0.$ Let
$k_0,j,k_1,...,k_q$ be non-negative integers with sum equal to
$D.$  Then
\begin{itemize}
\item [(i)] $$e^j(J^{[k_0]},I_1^{[k_1]},\dots, I_q^{[k_q]}; M)=e^j_{BR}({J};\overline{M}_t^*),$$
\noindent for any (FC)-sequence $x_1,...,x_t,$ with respect to
$(J,I_1,\ldots,I_q; M),$ of $t=k_1+...+k_q$ elements  consisting
of $k_1$ elements of $I_1$,..., $k_q$ elements of $I_q,$ where
$\overline{M}_t^*=M/((x_1,...,x_t)M:{\mathcal{I}}^{\infty}).$

\item [(ii)] If $k_0>0,$ then $e^j(J^{[k_0]},I_1^{[k_1]},\dots, I_q^{[k_q]};M)\neq 0,$
if and only if there exists an (FC)-sequence, with respect to
$(J,I_1,\ldots,I_q; M),$ of $t=k_1+...+k_q$ elements consisting of
$k_1$ elements of $I_1$,..., $k_q$ elements of $I_q.$
\end{itemize}
\end{thm}

The following result is an immediate consequence of item $(ii)$ of
the above theorem.

\begin{cor}\label{cor2}
Under the assumptions of Theorem \ref{Teo1.2} it follows that if
$k_0>0$ then the following statements are equivalent
\begin{itemize}
\item [(i)] $e^j(J^{[k_0]},I_1^{[k_1]},\dots, I_q^{[k_q]};M)\neq
0;$

\item [(ii)] $e^s(J^{[j+k_0-s]},I_1^{[k_1]},\dots,
I_q^{[k_q]};M)\neq 0$ for all $0\leq s\leq j+k_0;$

\item [(iii)] $e^s(J^{[j+k_0-s]},I_1^{[k_1]},\dots,
I_q^{[k_q]};M)\neq 0$ for some $0\leq s\leq j+k_0.$
\end{itemize}
\end{cor}

\section{Length of (FC)-sequences}\label{Section}

 This section gives characterizations for the length of maximal weak-(FC)-sequences
and the relation between maximal weak-(FC)-sequences and
reductions.
\begin{prop}\label{prop4}
Keeping the setup $(1),$ assume that ${\mathcal I}$  is not
contained in $\sqrt{\mbox{Ann}M}.$ Set $U=(J,I_1,\dots , I_q; M).$
Then the following statements hold.
\begin{itemize}
\item [(i)] Let $e^j(J^{[k_0]},I_1^{[k_1]},\dots, I_i^{[k_j]},\dots,I_q^{[k_q]};M)\neq 0$ and $k_i>0,\;i\geq
1.$ Suppose that $x\in I_i$ is a weak-(FC)-element with respect to
$U$, then $x$ is an (FC)-element.

\item [(ii)] Let $x_1,\ldots, x_t$ be a weak-(FC)-sequence in $I_1\cup\dots\cup I_q$
with respect to $U$. Then
$$\dim\left(\mbox{Supp} \left(M/(x_1,\ldots, x_t)M:\mathcal{I}^{\infty}\right)\right)\leq \dim\left(\mbox{Supp}
\left(M/0_{M}:\mathcal{I}^{\infty}\right)\right)-t$$ with equality if and only if
$x_1,\ldots, x_t$ is an (FC)-sequence of $G$ with respect to
$U.$

\item [(iii)] For any $1\leq i\leq q,$ the length of maximal weak-(FC)-sequences in $I_i$ with respect to $U$ is  an invariant.

\item [(iv)] For any $1\leq i\leq q,$ the length of maximal
(FC)-sequences in $I_i$ with respect to $U$ is  an invariant.

\item [(v)] If $x_1,\ldots, x_s$ is a maximal weak-(FC)-sequence
in $I_1\cup\dots\cup I_q$ with respect to $U$ and $\mbox{ht}(\mathcal{I}+\mbox{Ann}M/\mbox{Ann}M)=h>0$ then
$h\leq s$ and $x_1,\ldots, x_{h-1}$ is an (FC)-sequence.

\item [(vi)] If $\mbox{ht}(\mathcal{I}+\mbox{Ann}M/\mbox{Ann}M)=h>0$ and
$k_1+\cdots + k_q\leq h-1,$ then $e^j(J^{[k_0]},I_1^{[k_1]},\dots,I_q^{[k_q]};M)\neq 0.$

\item [(vii)] If
$\mbox{ht}(\mathcal{I}+\mbox{Ann}M/\mbox{Ann}M)=h>0$ and
$x_1,\ldots, x_t$ is a weak-(FC)-sequence  with respect to $U$ of
$t=k_1+\cdots + k_q\leq h-1$ elements consisting of $k_1$ elements
of $I_1,...,k_q$ elements of $I_q,$ then
$$e^j(J^{[k_0]},I_1^{[k_1]},\dots,I_q^{[k_q]};M)=e^j_{BR}
({J};\overline{M}),$$ \noindent where $\overline{M}=M/(x_1,\ldots,
x_t)M.$
\end{itemize}
\end{prop}

\begin{proof}
 We know by \cite[Theorem 4.1]{perez-bedregal1} that the
function $h(n,p,\bold r):=\ell\left(\frac{{\bf I}^{\bf
r}M_{n+p}}{J^{n}{\bf I}^{\bf r}M_{p}}\right)$ is, for all large
$n,p,\bold r,$ a polynomial of degree $D,$ which we denote by
$B(n,p,\bold r).$ Now, by \cite[eq. (4.7)]{perez-bedregal1}, we
have that for all large $n,p,\bold r,$

\begin{equation}\label{equation2}\ell\left(\frac{{\bf I}^{\bf r}\overline{M}^*_{n+p}}
{J^{n}{\bf I}^{\bf r}\overline{M}^*_p}\right)=\ell\left(\frac{{\bf
I}^{\bf r}M^*_{n+p}} {J^{n}{\bf I}^{\bf
r}M^*_p}\right)-\ell\left(\frac{{\bf I}^{{\bf
r}-\delta(i)}M^*_{n+p}} {J^{n}{\bf I}^{{\bf
r}-\delta(i)}M^*_p}\right),\end{equation} where $x$ is a
weak-(FC)-element in $I_i$ with respect to $U$ and
$\overline{M}=M/xM$ and
$\overline{M}^*=\overline{M}/0_{\overline{M}}:{\mathcal
I}^{\infty}=M/xM:{\mathcal I}^{\infty}.$

We first prove $(i).$ Since $e^j(J^{[k_0]},I_1^{[k_1]},\dots,
I_i^{[k_i]},\dots,I_q^{[k_q]};M)\neq 0$ and $k_i>0,\;i\geq 1,$ by
equality (\ref{equation2}) it follows that
$$\ell\left(\frac{{\bf I}^{\bf r}\overline{M}^*_{n+p}}
{J^{n}{\bf I}^{\bf r}\overline{M}^*_p}\right)$$ \noindent is a
polynomial of degree $D-1$  for all large $n,p,\bold r.$ Hence
$$\dim(\mbox{Supp}(M/xM:{\mathcal
I}^{\infty}))=D-1=\dim(\mbox{Supp}(M/0:{\mathcal I}^{\infty}))-1$$
and therefore $x$ is an (FC)-element.

For the proof of $(ii)$ see \cite[Proposition
4.5]{perez-bedregal1}.

We now prove $(iii).$ Notice that, by $(ii),$ the length of any
maximal weak-(FC)-sequence in $I_i$ with respect to $U$ is finite.
Now, by Remark \ref{remark0}, since $\mathcal I$ is not contained
in $\sqrt{\mbox {Ann}M},$ there exists $x_1,\ldots, x_{\ell}$  a
maximal weak-(FC)-sequence with respect to $U$ in $I_i.$ Let
$n,p,\bold r$ be large enough so that the function
$$\ell\left(\frac{{\bf I}^{\bf r}\overline{M}^*_{n+p}}
{J^{n}{\bf I}^{\bf r}\overline{M}^*_p}\right)$$ becomes a
polynomial, which we denote by $B^*(n, p,\bold {r}).$
 Fix an integer
$u\gg 0$ and set $p=n=r_1=\cdots =r_{i-1}=r_{i+1}=\cdots =r_q=u,$
 $B^*(r_i)=B^*(u,u,\ldots,r_i,\ldots,u)$ and
$B(r_i)=B(u,u,\ldots,r_i,\ldots,u).$ Then $B^*(r_i)$ and $B(r_i)$
are polynomials in $r_i.$ By equalities (\ref{equation2}) and
(\ref{equation0}) we have
\begin{equation}\label{equation3} B^*(r_i)=B(r_i)-B(r_i-1)
\end{equation}

 \noindent for all $r_i\geq u.$ Set $t=\deg
B(r_i).$ Since $\mathcal I$ is not contained in
$\sqrt{\mbox{Ann}M},$ $t\geq 0.$ We will prove, by induction on
$t,$ that $\ell=t+1$ and this will end the proof of $(iii).$ For
$t=0,$ we have by equality (\ref{equation3}) that $B^*(r_i)=0.$
From this follows that $\mathcal I$ is contained in $\sqrt
{\mbox{Ann}\overline{M}}.$ By Remark \ref{remark0} this implies
that $\ell=1=t+1.$ Since $\deg B^*(r_i) =t-1$ and $x_2,\ldots,
x_l$ is also a maximal weak-(FC)-sequence of $G$ with respect to
$(J,I_1,\dots,I_q; \overline{M}^*),$ by inductive assumption it
follows that $\ell-1=\deg B^*(r_i)+1=t.$ Thus $\ell=t+1$ and the
induction is complete.

Lets prove $(iv).$ Notice that by Theorem \ref{Teo1.2} $(ii)$ and
Corollary \ref{cor2}, the length of maximal (FC)-sequences with
respect to $U$ in $I_i$ is given by
$$\max\{k_i|e^0(J^{[k_0]},I_1^{[k_1]},\dots, I_q^{[k_q]};M)\neq 0\}.$$
\noindent Thus the length of maximal (FC)-sequences of $G$ with
respect to $U$ in $I_i$ is an invariant.

We prove now $(v).$ We will prove first that $h\leq s.$ Assume for
the contrary that $s<h.$ In this case
$$\mbox{ht}\left (\frac{\mathcal{I}+\mbox{Ann}(M/(x_1,\ldots, x_s)M)}{\mbox
{Ann}(M/(x_1,\ldots, x_s)M)}\right )>0.$$ Thus $\mathcal I$ is not
contained in $\sqrt{\mbox{Ann}(M/(x_1,\ldots, x_s) M)}$ and hence,
by Remark~\ref{remark0}, there is a weak-(FC)-element $x$ such
that $x_1,\ldots, x_s,x$ is a weak-(FC)-sequence  with respect to
$U$ in $I_1\cup \cdots \cup I_q.$ This contradict the maximality
of $x_1,\ldots, x_s$ and hence $h\leq s.$

Now, by $(ii)$ we have that
$$\dim(\mbox{Supp}(M/(x_1,\ldots, x_{h-1})M:\mathcal{I}^{\infty}))\leq \mbox
{dim}(\mbox{Supp}(M/0_{M}:\mathcal{I}^{\infty}))-(h-1).$$ But,
since $\mbox{ht}(\mathcal{I}+\mbox{Ann}M/\mbox{Ann}M)=h,$ we have
that
$$\mbox{ht}\left (\frac{\mathcal{I}+\mbox{Ann}(M/(x_1,\ldots, x_{h-1})M)}
{\mbox{Ann}(M/(x_1,\ldots, x_{h-1})M)}\right )>0.$$
\noindent
Therefore,
$$\dim(\mbox{Supp}\left (M/(x_1,\ldots, x_{h-1})M\right))=\dim(\mbox{Supp}\left
(M/(x_1,\ldots, x_{h-1})M:\mathcal{I}^{\infty}\right)).$$
Furthermore, since
$\mbox{ht}(\mathcal{I}+\mbox{Ann}M/\mbox{Ann}M)=h>0,$ we have that
$$\dim(\mbox{Supp}
M)=\dim(\mbox{Supp}(M/0_{M}:\mathcal{I}^{\infty})).$$ But clearly
$$\dim(\mbox{Supp}\left (M/(x_1,\ldots, x_{h-1})M\right))\geq \dim(\mbox{Supp}M)-
(h-1).$$ Taking into account all the above facts, we get
$$\dim(\mbox{Supp}\left(M/(x_1,\ldots, x_{h-1})M:\mathcal{I}^{\infty}\right))= \dim(\mbox{Supp}\left(M/0_{M}:\mathcal{I}^{\infty}\right))-(h-1).$$
\noindent Hence by $(ii),$ $x_1,\ldots, x_{h-1}$ is an
(FC)-sequence with respect to $U.$

Lets prove $(vi).$ By $(v)$ and the assumption that $k_1+\cdots +
k_q\leq h-1,$ it has been proved that there exists an
(FC)-sequence with respect to $U$ in $I_1\cup \cdots I_q$
consisting of $k_1 $ elements of $I_1,...,k_q$ elements of $I_q.$
Hence by \cite[Theorem 4.7]{perez-bedregal1}
$e^j(J^{[k_0]},I_1^{[k_1]},\dots, I_q^{[k_q]};M)\neq 0.$

The proof of $(vii)$ follows immediately from $(v)$ and
 \cite[Theorem 4.7]{perez-bedregal1}.
\end{proof}

 Set $U=(J,I_1,...,I_q; M)$ and let
$L_U(M)$ denote the set of lengths of maximal weak-(FC)-sequences
in $I_1\cup \cdots \cup I_q$ with respect to $U.$ Based on the
equality (\ref{equation3}), we come to an important
characterization for the lengths of maximal weak-(FC)-sequences.

\begin{prop}\label{prop5}
 In the setup $(1),$ assume that $\cal I$ is not contained in \break $\sqrt{\mbox{Ann}(M)}.$ Set $r=r_i$; $Q(r)=B
(u,u,\ldots,r_i,\ldots,u)$ and $Q({\bf r})=B (u,u,\bf{r})$.
Suppose that $s$ is the length  of maximal weak-(FC)-sequences em
$I_i$ with respect to $U$. Then the following statements hold.
\begin{itemize}
\item [(i)] $s=\deg(Q(r))+1.$ \item [(ii)]
$\max L_U(M)=\deg(Q({\bf r}))+1.$
\end{itemize}
\end{prop}

\begin{proof} First note that $Q(r)$ and $Q({\bold r})$ do not
depend on $u,$ for large $u.$ Notice that $(i)$ follows by the
proof of Proposition \ref{prop4} $(iii).$
 Now assume that
$\deg(Q({\bold r}))=l$. Since ${\mathcal I}$ is not contained in
$\sqrt{\mbox{Ann}M}$, we have by Proposition~\ref{Obs1} that there
exists a weak-(FC)-element  $x\in I_i$ with respect to $U$. Set
$\overline{M}=M/xM$ and $\overline{M}^*=M/(xM:{\cal I}^{\infty}).$

From the proof of Proposition \ref{prop4} we get
\begin{equation}\label{equation8}
\begin{array}{lll} \ell\left(\frac{{\bf I}^{\bf
r}\overline{M}^*_{u+u}} {J^{u}{\bf I}^{\bf
r}\overline{M}^*_u}\right)&=&\ell\left(\frac{{\bf I}^{\bf
r}M^*_{u+u}} {J^{u}{\bf I}^{\bf
r}M^*_u}\right)-\ell\left(\frac{{\bf I}^{{\bf
r}-\delta(i)}M^*_{u+u}}
{J^{u}{\bf I}^{{\bf r}-\delta(i)}M^*_u}\right)\\
&=&B(u,u,{\bold r})-B(u,u,\bold r-\delta(i))\\
&=&Q({\bold r})-Q(\bold r-\delta(i)).
\end{array}
\end{equation}
Set $Q^*({\bold r})=Q({\bold r})-Q({\bold r}-\delta(i)).$ Hence
$\deg(Q^*({\bold r}))=l-1.$

 In order to prove $(ii)$, We
first prove the inequality $$\max L_U(M)\geq \deg(Q({\bold
r}))+1$$ by induction on $l=\deg(Q({\bold r})).$ For $l=0$ the
result trivially holds since $\max L_U(M)\geq 1.$
   Let ${\mathcal U}=(J, I_1,\dots,
I_q; \overline{M}^*)$. By inductive assumption and equation
(\ref{equation8}) it follows that $\max L_{\mathcal U}(J,
I_1,\dots, I_q; \overline{M}^*)\geq l$ and there exists
$x_1,...,x_l$ a weak-(FC)-sequence  with respect to ${\mathcal U}$
in $I_1\cup \cdots \cup I_q.$ From this follows that
$x,x_1,...,x_l$ is a weak-(FC)-sequence in $I_1\cup \cdots \cup
I_q$ with respect to $U$. Hence
$$\max L_U(M)\geq l+1=\deg(Q({\bold r}))+1.$$
\noindent The induction is complete.

To conclude the proof of $(ii),$ we prove now the inequality
$$\max L_U(M)\leq \deg(Q({\bold r}))+1$$ by induction on
$t=\max L_U(M).$ Since ${\mathcal I}$ is not contained in
$\sqrt{\mbox{Ann}M}$, $t\geq 1$. The case $t=1$,  is trivial.

Let $x_1,...,x_t$ be an arbitrary maximal weak-(FC)-sequence in
$I_1\cup \cdots \cup I_q$ with respect to $U$. Without loss of
generality we may assume that $x_1\in I_i$. By equality
(\ref{equation8}) we have $\deg(Q^*({\bold r}))\leq l-1.$  Then
$x_2,...,x_t$ is also a maximal weak-(FC)-sequence  with respect
to ${\mathcal U}$. By inductive assumption, it follows that
$t-1\leq \deg(Q^*({\bold r}))+1\leq l.$ Thus, $t\leq l+1.$ The
induction is complete.
\end{proof}

\begin{lem}\label{lema1}
In the setup $(1),$ set $U_1=(J_1,J_2,I_1,...,I_q; M)$ and
$U_2=(J_1,I_1,...,I_q; M)$ where  $J_1$ and $J_2$ are finitely
generated $R$-submodules of $G_1$ of finite colength. Then for any
$1\leq i\leq q$ we have the following relations.
\begin{itemize}
\item [(i)] If $l_1$ and $l_2$ are lengths  of maximal
weak-(FC)-sequences in $I_j$ with respect to
$U_1$ and $U_2$, respectively, then $l_1=l_2.$
\item
[(ii)] If $l$ and $f$ are lengths  of maximal (FC)-sequences in
$I_j$ with respect to $U_1$ and $U_2$,
respectively, then $l=f.$
\end{itemize}
\end{lem}
\begin{proof}
The proof of $(i):$ Since $x_1,...,x_{l_1}$ is a maximal
weak-(FC)-sequence in $I_i$ with respect to $U_1$, then by Remark
\ref{remark0}, ${\mathcal I}\subseteq
\sqrt{\mbox{Ann}(M/(x_1,...,x_{l_1})M)}$. By definition of a
weak-(FC)-sequence and Lemma \ref{lema0}, $x_1,...,x_{l_1}$ is
also a maximal weak-(FC)-sequence in $I_i$ with respect to $U_2.$
By Proposition~\ref{prop4} $(iii)$ we get $l_1=l_2.$

The proof of $(ii)$: Note that, by Lemma \ref{lema0}, any
(FC)-sequence in $I_i$ with respect to $U_1$ is also an
(FC)-sequence in $I_i$ with respect to $U_2$. Hence $l\leq f.$
Assume that $x_1,...,x_{l_1}$ is  a maximal weak-(FC)-sequence in
$I_i$ with respect to $U_1.$ By $(i),$ $x_1,...,x_{l_1}$ is also a
maximal weak-(FC)-sequence in $I_i$ with respect to $U_2.$ By
Proposition \ref{prop4} $(i),$ there exists $f<l_1$ elements, say
$x_1,...,x_f$, amongst the set $x_1,...,x_{l_1}$ which form  a
maximal (FC)-sequence in $I_i$ with respect to $U_2.$ By
Proposition \ref{prop4} $(i),$ we have
$$\dim(\mbox{Supp}(M/((x_1,...,x_f):{\mathcal
I}^{\infty})))=\dim(\mbox{Supp}(M/(0:{\mathcal I}^{\infty})))-f.$$
By Proposition \ref{prop4} $(ii)$ and Lemma \ref{lema0},
$x_1,...,x_f$ is also an (FC)-sequence in $I_i$ with respect to
$U_1.$ Thus $f\leq l$ and we get the result.
\end{proof}

Let $\mu({J})$ denote the minimal number of generators of an
$R$-submodule ${ J}$ of $G_1$.

\begin{defn} Let $I_1,...,I_q$ be $R$-submodules of $G_1.$ For
$i=1,...,q$, an $R$-submodule ${ J}_i$ of $I_i$ is called a
reduction of $I_i$ with respect to $U'=(I_1,...,I_q;M)$ if
$${\bf I}^{{\bf r}}M_p={J}_i{\bf I}^{{\bf r}-\delta(i)}M_p,  \mbox{ for all large } {\bf
r}, {\mbox { and all }} p\geq 0.$$ Set $N_{U'}(I_i)=\min\{\mu({
J}_i)|{ J}_i \mbox{ is a reduction of } I_i\mbox{  with respect to
}U'\}.$ \break $N_{(I;M)}(I)$ will be denoted by $N(I).$
\end{defn}

Let $I$ be an $R$-submodule of $G_1.$  Let ${\mathcal
R}[I]:=\oplus_{n\in \N}{ I}^n$ be the  graded $R$-subalgebra of
$G$ generated in degree one by $I$. We call this algebra the {\it
Rees algebra} of ${I}.$ The number $s({I}):=\dim
\left(\frac{{\mathcal R}[I]}{{\mathfrak m}{\mathcal
R}[I]}\right)$, is called the {\it analytic spread} of $I.$ More
generally, if $I_1,...,I_q$ are $R$-submodules of $G_1$ and
$I=I_1\cdots I_q$ we define the Rees algebra of $I,$ also denoted
by ${\mathcal R}[I]$, as the $R$-subalgebra of $G^{(q)}$ given by
$${\mathcal R}[I]=\oplus_{n\in \N}{ I}^n=\oplus_{n\in \N}
I_1^n\cdots I_q^n,$$ where $G^{(q)}$ is the standard graded
$R$-algebra given by $G^{(q)}=\oplus_{n\in \N}G_{qn}.$
Analogously, we call the number $s({I}):=\dim
\left(\frac{{\mathcal R}[I]}{{\mathfrak m}{\mathcal
R}[I]}\right)$, the {\it analytic spread} of $I.$ Notice that,
since $\frac{{\mathcal R}[I]}{{\mathfrak m}{\mathcal R}[I]}$ is
generated over the field $R/\mathfrak m$ by $I/{\mathfrak m}I,$ we
have that $s(I)\leq \mu(I),$ where $\mu(I)$ is the minimal number
of generators of $I.$

Now we will describe the length of maximal weak-(FC)-sequences and
also the relation between maximal weak-(FC)-sequences and
reductions.

\begin{thm}\label{teo1}
Let $(J_1,\dots, J_t)$ be finitely generated $R$-submodules of
$G_1$ of finite colength. Set $U=(J_1,...,J_t,I_1,...,I_q;M)$ and
$U'=(I_1,...,I_q;M).$
 For any
$1\leq i\leq q$ set
$$\hat{\bf I}_{i}=I_1\cdots I_{i-1}I_{i+1}\cdots I_q$$ if $q>1$
and $\hat{\bf I}_{i}=G_1$ if $q=1;$ set $R_i={\cal R}[I_i].$ Then
\begin{itemize}
\item [(i)] For any $1\leq i\leq q$, the length  of maximal
weak-(FC)-sequences in $I_i$ with respect to $U$ is an invariant
and this invariant does not depend on $t$ and $J_1,\dots, J_t.$

\item [(ii)] If $l$ is the length  of
maximal (FC)-sequences in $I_i$ with respect to $U,$ then
$$l=\dim\left(\mbox{Proj}\left(\frac{R_i}{\left({\mathfrak
m}^{u}{\hat{\bf I}}_i^uM_{u+u}{R_i}:_{R_i}{\hat{\bf
I}}_i^uM_{u+u}R_i\right)}\right)\right)+1\leq s(I_i),$$ for all
large $u.$

\item [(iii)] If $x_1,...,x_l$ is  a
maximal weak-(FC)-sequence in $I_i$ with respect to $U,$ then
${\mathfrak J}=(x_1,...,x_l)$ is a reduction of $I_i$ with respect
to $U'$ and $l=N_{U'}(I_i).$

\item [(iv)] $\max L_U(M)=\dim\left
(\mbox{Proj}\left(\frac{R[I]}{({
J}^{u}M_{u}R[I]:M_{tu+u}R[I])}\right)\right )+1\leq s(I),$ where
$I=I_1\cdots I_q$ and $J=J_1\cdots J_t.$

\item [(v)] $\mbox{ht}({\cal I}+\mbox{Ann}(M)/\mbox{Ann}(M))\leq \max L_{U}(M),$ where $\cal I$ is the ideal in $G$ generated by
$I_1\cdots I_q.$
\end{itemize}
\end{thm}
\begin{proof}
The proof of $(i)$: By Proposition \ref{prop4} $(iii),$ it follows
that for any $1\leq i\leq q$, the length of maximal
weak-(FC)-sequences in $I_i$ with respect to $U$ is an invariant
of $I_i.$ Let $J_1'$ and $J_1''$ be finitely generated
$R$-submodules of $G_1$ of finite colength. We call $l_1,l_2,l_3$
the lengths of maximal weak-(FC)-sequences in $I_i$ with respect
to $(J_1', J_1'',J_2,\dots, J_t,I_1,\dots, I_q;M),\;\;
(J'_1,\dots, J_t,I_1,\dots, I_q;M)$ and  $(J''_1,\dots,
J_t,I_1,\dots, I_q;M),$ respectively. From Lemma \ref{lema1} $(i)$
we have $l_2=l_1=l_3.$ Thus, $l_2$ does not depend on $t$ and
$J_1,\dots,J_t$.

The proof of $(ii):$ Set $J=J_1\cdots J_t$. Assume that
$$B(n,p,\bold r):=\ell\left(\frac{{\bf I}^{\bf r}M_{tn+p}}{J^{n}{\bf I}^{\bf r}M_{p}}\right).$$
\noindent By item $(i),$ $l$ is independent of $t$ and $J.$ Then
suppose that $t=1$ and $J={\mathfrak m}G_1.$ In this case, since
$M$ is generated in degree zero, we have
$$B(n,p,\bold r)=\ell\left(\frac{{\bf I}^{\bf r}M_{n+p}}{{\mathfrak m}^{n}{\bf I}^{\bf r}M_{n+p}}\right).$$
\noindent
 For all $n,p, {\bold r}\gg 0$. Fix an integer $u\gg 0$, and
set $p=n=r_1=\cdots=r_{i-1}=r_{i+1}=\cdots =r_q=u$; $r_i=r$ and
$Q(r)=B(u,u...,r,...,u).$ Then $Q(r)$ is a polynomial in $r$. By
Proposition \ref{prop5} $(i)$ we have $l=\deg(Q(r))+1$. Set
$$K_u=\bigoplus_{r\in \N}\frac{{\hat{\bf I}}_i^uI_i^rM_{u+u}}{{\mathfrak
m}^u{\hat{\bf I}}_i^uI_i^rM_{u+u}}=:\bigoplus_{r\in \N}[K_u]_r.$$
\noindent It is easily seen that $K_u$ is a finitely generated
graded $R_i$-module. Furthermore,
$$\ell([K_u]_r)=\ell\left(\frac{{\hat{\bf I}}_i^uI_i^rM_{u+u}}{{\mathfrak
m}^u{\hat{\bf I}}_i^uI_i^rM_{u+u}}\right)=Q(r)$$ is a polynomial
in $r$ having degree
$$\dim\left(\mbox{Supp}(K_u)\right)=
\dim\left(\mbox{Proj}\left(\frac{R_i}{\mbox{Ann}_{R_i}(K_u)}\right)\right).$$
Therefore
$l=\dim\left(\mbox{Proj}\left(\frac{R_i}{\mbox{Ann}_{R_i}(K_u)}\right)\right)+1$
for all large $u.$ Observe that
$$
\begin{array}{lll}
\dim\left(\mbox{Proj}\left(\frac{R_i}{\mbox{Ann}_{R_i}(K_u)}\right)\right)
&=&\dim\left(\mbox{Proj}\left(\frac{R_i}{\left({\mathfrak
m}^{u}{\hat{\bf I}}_i^uM_{u+u}{R_i}:_{R_i}{\hat{\bf I}}_i^uM_{u+u}R_i\right)}\right)\right)\\
&\leq & \dim\left(\mbox{Proj}\left(\frac{R_i}{{\mathfrak m}^uR_i}\right)\right)\\
&=&s(I_i)-1.
\end{array}
$$
\noindent for all large $u$. Thus
$$l\leq s(I_i).$$
The proof  of $(iii)$: Let $x_1,...,x_{l}$ be  a maximal
weak-(FC)-sequence in $I_i$ with respect to $U.$ By Remark
\ref{remark0} we have that
$$\mbox{ht}\left(\frac{{\mathcal
I}+\mbox{Ann}(M/(x_1,...,x_{l})M)}{\mbox{Ann}(M/(x_1,...,x_{l})M)}\right)=0.$$
\noindent Hence $\sqrt{\mathcal{I}}=\sqrt{\mbox{Ann}\left
(\frac{M}{(x_1,...,x_{l})M}\right )}.$

We will prove next by induction on $t$ that
$$(x_1,...,x_t)M_{p+|{\bf r}|-1}\cap {\bf I}^{{\bf r}}M_p=
(x_1,...,x_t){\bf I}^{{\bf r}-\delta(i)}M_p$$ \noindent for all
large ${\bold r}$, all $p$ and all $t\leq l$. If $t=0$ the result
trivially holds.  Set $N=(x_1,...,x_{t-1})M:{{\mathcal
I}^{\infty}}\subseteq M$. Since $x_t$ satisfies the condition
$(FC)_1$ with respect to $(I_1,\dots, I_q;M/(x_1,...,x_{t-1})M)$,
$$(N_{|\bold r|+p}+x_tM_{p+|\bold r|-1})\cap ({\bf I}^{{\bf
r}}M_p+N_{|\bold r|+p}) = x_t{\bf I}^{{\bf
r}-\delta(i)}M_p+N_{|\bold r|+p}$$ \noindent for all large $r_i$
and all non-negative integers $p,r_1,...,r_{i-1},r_{i+1},...,r_q$
and all $t\leq l.$ From this we get that \small{$$
\begin{array}{lll}
\vspace{0.3cm} (N_{|\bold r|+p}+x_tM_{p+|\bold r|-1})\cap {\bf
I}^{{\bf r}}M_p
&=&{\bf I}^{{\bf r}}M_p\cap ({\bf I}^{{\bf r}}M_p+N_{|\bold r|+p})\cap (N_{|\bold r|+p}+x_tM_{p+|\bold r|-1})\\
\vspace{0.3cm} &=&{\bf I}^{{\bf r}}M_p\cap (x_t{\bf I}^{{\bf
r}-\delta(i)}M_p+N_{|\bold r|+p})
\end{array}
$$}
\noindent for all large ${\bold r}$ and all $p$. Therefore,
\begin{equation}\label{equation5}
(N_{|\bold r|+p}+x_tM_{p+|\bold r|-1})\cap {\bf I}^{{\bf r}}M_p
=x_t{\bf I}^{{\bf r}-\delta(i)}M_p+N_{|\bold r|+p}\cap {\bf
I}^{{\bf r}}M_p
\end{equation}
\noindent for all large ${\bold r}$ and all $p$. By the Artin-Rees
Lemma \cite[Lemma 3.5]{Bedregal-Perez}, there exists $\bf c$ such
that
$$N_{|\bold r|+p}\cap {\bf I}^{{\bf r}}M_p={\bf I}^{{\bf r}-{\bf c}}({\bf I}^{{\bf c}}M_p\cap
N_{|{\bf c}|+p})$$ \noindent for all large ${\bold r}\geq {\bf
c}$.

Hence, since by definition of $N$ we have that ${\bf I}^{{\bf
r}}N_p\subseteq (x_1,...,x_{t-1})M_{|\bold r|+p-1}$ for large
${\bold r}$ and all $p,$ we get
$$\begin{array}{lll}
N_{|\bold r|+p}\cap {\bf I}^{{\bf r}}M_p&=& {\bf I}^{{\bf r}-{\bf
c}}({\bf I}^{{\bf c}}M_p\cap
N_{|{\bf c}|+p})\\
&\subseteq &{\bf I}^{{\bf r}}M_p\cap {\bf I}^{{\bf r}-{\bf
c}}N_{|{\bf c}|+p}\\
& \subseteq & {\bf I}^{{\bf r}}M_p\cap (x_1,...,x_{t-1})M_{|\bold
r|+p-1}
\end{array}
$$
\noindent for all large ${\bold r}$ and all $p$. Therefore,
\begin{equation}\label{equation6}
N_{|\bold r|+p}\cap {\bf I}^{{\bf r}}M_p={\bf I}^{{\bf r}}M_p\cap
(x_1,...,x_{t-1})M_{|\bold r|+p-1}
\end{equation}
\noindent for all large ${\bold r}$ and all $p$. By inductive
assumption we see that $${\bf I}^{{\bf r}}M_p\cap
(x_1,...,x_{t-1})M_{|\bold r|+p-1}= (x_1,...,x_{t-1}){\bf I}^{{\bf
r}-\delta(i)}M_p$$ \noindent for all large ${\bold r}$ and all
$p$. Thus, by equality (\ref{equation6}), we have $$N_{|\bold
r|+p}\cap {\bf I}^{{\bf r}}M_p=(x_1,...,x_{t-1}){\bf I}^{{\bf
r}-\delta(i)}M_p$$ \noindent for all large ${\bold r}$ and all
$p$. Hence combining this fact with equality (\ref{equation5}) we
get \small{$$\begin{array}{lll} \vspace{0.3cm}(N_{|\bold
r|+p}+x_tM_{|\bold r|+p-1})\cap {\bf I}^{{\bf r}}M_p
&=&(x_1,...,x_{t-1}){\bf I}^{{\bf r}-\delta(i)}M_p+x_t{\bf I}^{{\bf r}-\delta(i)}M_p\\
\vspace{0.3cm} &=&(x_1,...,x_t){\bf I}^{{\bf r}-\delta(i)}M_p
\end{array}
$$}
\noindent for all large ${\bold r}$ and all $p$.  It follows
directly from this that
\begin{equation}\label{equation7}
 (x_1,...,x_t)M_{|{\bf r}|+p-1}\cap {\bf
I}^{{\bf r}}M_p
 =(x_1,...,x_t){\bf I}^{{\bf r}-\delta(i)}M_p
\end{equation}
\noindent for all large ${\bold r}$ and all $p$. The induction is
complete.

Since $\sqrt{\mathcal{I}}=\sqrt{\mbox{Ann}\left
(\frac{M}{(x_1,...,x_{l})M}\right )}.$
$${\bf I}^{{\bf
r}}M_p\subseteq (x_1,...,x_l)M_{|{\bf r}|+p-1}$$ \noindent for all
large ${\bold r}$ and all $p$. Hence
$$
\begin{array}{lll}
\vspace{0.3cm} {\bf I}^{{\bf r}}M_p&=&(x_1,...,x_l)M_{|{\bf
r}|+p-1}\cap {\bf I}^{{\bf r}}M_p\\
\vspace{0.3cm} &=&(x_1,...,x_l){\bf I}^{{\bf r}-\delta(i)}M_p
\end{array}
$$ \noindent for all large ${\bold r}$ and all $p$. Therefore $(x_1,...,x_l)$
is a reduction of $I_i$ with respect to $(I_1,..,I_q;M).$

Now, we prove that $l=N_{U'}(I_i)$. So far we have proved that
$N_{U'}(I_i)\leq l.$ Let us assume that $N_{U'}(I_i)<l,$ that is,
there exists ${\mathfrak J}_i=(x_1,...,x_t)\;\;\;(t<l)$ which is a
reduction of $I_i$ with respect to $(I_1,..,I_q;M)$. Let
$y_1,...,y_k$ be  a maximal weak-(FC)-sequence in ${\mathfrak
J}_i$ with respect to $(J_1,...,J_t,I_1,\dots,I_i,{\mathfrak
J}_i,I_{i+1},...,I_q;M)$

 Set $R_i^*=R[{\mathfrak J}_i]$. By $(ii)$ we
have
$$k=\dim\left(\mbox{Proj}\left(\frac{R_i^*}{\left({\mathfrak
m}^{u}{\hat{\bf I}}_i^uM_{u+u}{R_i^*}:_{R_i^*}{\hat{\bf
I}}_i^uM_{u+u}R_i^*\right)}\right)\right)+1\leq s({\mathfrak
J}_i),$$ for all large $u.$ But, since $s({\mathfrak J}_i)\leq
\mu({\mathfrak J}_i),$ we have that $k\leq t$.

 Set $J=J_1\cdots J_t$. Assume that
$$B(n,p,\bold r):=\ell\left(\frac{{\bf I}^{\bf r}M_{tn+p}}{J^{n}{\bf I}^{\bf r}M_{p}}\right).$$
\noindent for all $n,p,\bold r \gg 0.$ By item $(i)$ $l$ is
independent of $t$ and $J.$ Then we may suppose that $t=1$ and
$J={\mathfrak m}G_1.$ In this case, since $M$ is generated in
degree zero, we have
$$B(n,p,\bold r)=\ell\left(\frac{{\bf I}^{\bf r}M_{n+p}}{{\mathfrak m}^{n}{\bf I}^{\bf r}M_{n+p}}\right).$$
\noindent
 For all $n,p, {\bold r}\gg 0$. Fix an integer $u\gg 0$, and
set $p=n=r_1=\cdots=r_{i-1}=r_{i+1}=\cdots =r_q=u$; $r_i=r$ and
$Q(r)=B(u,u...,r,...,u).$ Then $Q(r)$ is a polynomial in $r$.

 By
Proposition \ref{prop5} $(i)$ we have $l=\deg(Q(r))+1$. Set
$$K_u=\bigoplus_{r\in \N}\frac{{\hat{\bf I}}_i^uI_i^rM_{u+u}}{{\mathfrak
m}^u{\hat{\bf I}}_i^uI_i^rM_{u+u}}=:\bigoplus_{r\in \N}[K_u]_r.$$
\noindent It is clear that $K_u$ is a finitely generated graded
$R_i$-module. Furthermore, for large $u,$
$$\ell([K_u]_r)=\ell\left(\frac{{\hat{\bf I}}_i^uI_i^rM_{u+u}}{{\mathfrak
m}^u{\hat{\bf I}}_i^uI_i^rM_{u+u}}\right)=Q(r).$$ Set
$$N^*_u=\bigoplus_{r\in \N}\frac{{\hat{\bf I}}_i^uI_i^z{\mathfrak J}_i^rM_{u+u}}{{\mathfrak
m}^u{\hat{\bf I}}_i^uI_i^z{\mathfrak
J}_i^rM_{u+u}}=:\bigoplus_{r\in \N}[N^*_u]_r.$$ \noindent It is
clear that $N^*_u$ is a finitely generated graded $R^*_i$-module.

 Since ${\mathfrak J}_i$ is a reduction of
$I_i$ with respect to $(I_1,...,I_q;M)$, there exists $c$ such
that
$$[N^*_u]_r=\frac{{\hat{\bf I}}_i^uI_i^z{\mathfrak J}_i^rM_{u+u}}{{\mathfrak
m}^u{\hat{\bf I}}_i^uI_i^z{\mathfrak
J}_i^rM_{u+u}}=\frac{{\hat{\bf I}}_i^uI_i^{r+z}M_{u+u}}{{\mathfrak
m}^u{\hat{\bf I}}_i^uI_i^{r+z}M_{u+u}}=[K_u]_{r+z}$$ \noindent for
all $z\geq c$. Thus, we get
$Q(r+z)=\ell([K_u]_{r+z})=\ell([N^*_u]_r).$

Hence
$$
\begin{array}{lll}
\vspace{0.3cm}
l=\deg(Q(r))+1&=&\deg(Q(r+z))+1=\dim(\mbox{Supp}(K_u^*))+1\\
\vspace{0.3cm}
&=&\dim\left(\mbox{Proj}\left(\frac{R^*_i}{\left({\mathfrak
m}^u{\hat{\bf I}}_i^uI_i^zM_{u+u}R_i^*:{\hat{\bf
I}}_i^iI_i^zM_{u+u}R_i^*\right)}\right)\right)+1
\end{array}
$$
It is completely clear that $$\left({\mathfrak m}^u{\hat{\bf
I}}_i^uM_{u+u}R_i^*:{\hat{\bf
I}}_i^uM_{u+u}R_i^*\right)\subseteq\left({\mathfrak m}^u{\hat{\bf
I}}_i^uI_i^zM_{u+u}R_i^*:{\hat{\bf
I}}_i^uI_i^zM_{u+u}R_i^*\right).$$ From the above facts and $(i)$
we get
$$
\begin{array}{lll}
\vspace{0.3cm}
l&=&\dim\left(\mbox{Proj}\left(\frac{R^*_i}{\left({\mathfrak
m}^u{\hat{\bf I}}_i^uI_i^zM_{u+u}R_i^*:{\hat{\bf
I}}_i^iI_i^zM_{u+u}R_i^*\right)}\right)\right)+1\\
\vspace{0.3cm}
&\leq&\dim\left(\mbox{Proj}\left(\frac{R^*_i}{\left({\mathfrak
m}^u{\hat{\bf I}}_i^uM_{u+u}R_i^*:{\hat{\bf
I}}_i^iM_{u+u}R_i^*\right)}\right)\right)+1\\
\vspace{0.3cm}  &=&k.
\end{array}$$
\noindent Thus, $t\geq k\geq l$. This contradict the assumption
that $t<l,$ and hence completes the proof of $(iii).$

 The proof of $(iv):$ Set
$J=J_1\cdots J_t$. Since $J$ has finite colength in $G_t$, there
exists an integer $k$ such that ${\frak m}^kG_t\subseteq J.$

Assume that
$$B(n,p,\bold r):=\ell\left(\frac{{\bf I}^{\bf r}M_{tn+p}}{J^{n}{\bf I}^{\bf r}M_{p}}\right).$$ \noindent is a polynomial for all
$p,n, {\bold r}\geq v$. Fix an integer $u\geq v$, by $(i)$ and
Proposition \ref{prop5} $(ii)$ it follows that $\max
L_U(M)=\deg(B(u,u,\bold r))+1$. Set $r_1=...=r_q=m.$ We get a
polynomial
$$B(u,u,m,...,m)=\ell\left(\frac{{\bf I}^{m}M_{tu+u}}{J^{u}{\bf I}^mM_{u}}\right).$$
\noindent It is  clear that $$\deg(B(n,p,{\bold
r}))=\deg(B(u,u,m,...,m))=\dim\left
(\mbox{Proj}\left(\frac{R[I]}{({
J}^{u}M_{u}R[I]:M_{tu+u}R[I])}\right)\right ).$$ \noindent From
this follows that $$\max L_U(M)=\dim\left
(\mbox{Proj}\left(\frac{R[I]}{({
J}^{u}M_{u}R[I]:M_{tu+u}R[I])}\right)\right )+1$$ for all large
$u$.

Notice that
$$\begin{array}{lll}
\vspace{0.3cm} \dim\left (\mbox{Proj}\left(\frac{R[I]}{({
J}^{u}M_{u}R[I]:M_{tu+u}R[I])}\right)\right )&\leq &\dim\left
(\mbox{Proj}\left(\frac{R[I]}{({\frak
m}^{ku}M_{tu+u}R[I]:M_{tu+u}R[I])}\right)\right )\\
\vspace{0.3cm} &\leq & \dim\left
(\mbox{Proj}\left(\frac{R[I]}{({\frak m}^{ku}R[I]}\right)\right
)\\
\vspace{0.3cm} &=& s(I)-1. \end{array}
$$
Thus, we get $(iv).$

The proof of $(v)$: From Proposition \ref{prop4} $(v)$ we get
$(v).$
\end{proof}

\begin{rem} {\rm Let $I$ be a finitely generated
$R$-submodule of $G_1$, assume that the ideal ${\mathcal I}$ of
$G$ generated by $I$  is not contained in $\sqrt{\mbox{Ann}(G)}$
and let $J$ be a finitely generated $R$-submodule of $G_1$ of
finite colength. Let $x_1,...,x_l$ be a maximal weak-(FC)-sequence
in $I$ with respect to $(J,I;G)$. Set ${\mathfrak
J}=(x_1,...,x_l).$ By Theorem \ref{teo1}, items $(iii)$ and
$(iv)$, ${\mathfrak J}$ is a reduction of $I$ (for ideals see
\cite{Viet9}) and
$$l=N({I})=\mu({\mathfrak
J})=s({I}).$$}
\end{rem}

Thus, as an immediate consequence of Theorem \ref{teo1}, we get
the following interesting result.

\begin{thm} \label{teo2} Let $I$ be a finitely generated
$R$-submodule of $G_1$, assume that ${\mathcal I}$ is not
contained in $\sqrt{\mbox{Ann}(G)}$ and let $J$ be a finitely
generated $R$-submodule of $G_1$ of finite colength. Suppose that
$l$ is the length of maximal weak-(FC)-sequences in $I$ with
respect to $(J,I;G).$ Then
\begin{itemize}
\item [(i)] $l=N(I)=s({I}).$
\item [(ii)] If
${\mathfrak J}$ is generated by a maximal weak-(FC)-sequence in
${I}$ with respect to $(J,I).$ Then ${\mathfrak J}$ is a reduction
of $I$ and $\mu({\mathfrak J})=l=s(I).$
\end{itemize}
\end{thm}

Now, we are interested in determine the length of maximal
(FC)-sequences and its relation with mixed multiplicities. In
order to explain some  notations of the next result, we need to
define a slightly different kind of mixed multiplicities which are
needed for the remaining of this section. Let $J$ as before be a
finitely generated $R$-submodule of $G_1$ of finite colength and
let $K$ be an $R$-submodule of $G_q,\, q\geq 1$. Then, as in the
proof of \cite[Theorem 4.1]{perez-bedregal1}, we can prove that
the length
$$\ell\left(\frac{K^nM_{p+m}}{J^mK^nM_p}\right)$$
is, for large $n,m,p$ a polynomial of degree
$s:=\dim{\mbox{Supp}}\left (M/(0_M:{\cal K}^{\infty})\right ),$
where $\cal K$ is the ideal of $G$ generated by $K.$ The terms of
higher degree of this polynomial could be written as
$$S(m,p,n)=\sum_{k_0+i+j=s}\frac{e^j(J^{[k_0]},K^{[i]};M)}{k_0!j!i!}
m^{k_0}n^{i}p^j.$$

\begin{thm}\label{teo3}
Keeping the setup $(1)$, set $I=I_1\cdots I_q$ and
$U=(J,I_1,...,I_q;M)$. Let $\max L^*_{U}(M)$ denote the set of
lengths of maximal (FC)-sequences in $I_1\cup\dots \cup I_q$ with
respect to $U.$ Suppose that $l$ is the length of maximal
(FC)-sequences in $I$ with respect to $({\mathfrak m}G_1, I;M).$
Then
\begin{itemize}
\item [(i)] If $e^j(J^{[k_0]},I_1^{[k_1]},\dots, I_q^{[k_q]};M)\neq
0$, then $e^j(J'^{[m_0]},I_1^{[m_1]},\dots, I_q^{[m_q]};M)\neq 0$
for any  finitely generated $R$-submodule $J'$ of $G_1$ of finite
colength and for all $m_1\leq k_1, ...,m_q\leq k_q$,\, $m_0> 0$
such that $j+m_0+m_1+...+m_q=D.$

\item [(ii)]
$$e^j(J^{[D-j-i]},I^{[i]};M)=
\sum_{|\bold k|=i}(i!)\frac{e^j(J^{[D-j-i]},I_1^{[k_1]},\dots,
I_q^{[k_q]};M)}{k_1!\cdots k_q!},$$ \noindent  for all $i\leq
D-j.$

\item [(iii)]$\max L^*_U(M)=l\leq s({ I})-1$.
\item [(iv)] $\mbox{ht}({\mathcal
I}+\mbox{Ann}(M)/\mbox{Ann}(M))-1\leq \max L^*_U(M).$
\end{itemize}
\end{thm}
\begin{proof} The proof of $(i)$: the proof is by induction on
$t=|\bold k|.$ For $t=0$, by  \cite[Lemma 4.3]{perez-bedregal1} we
have $e^j(J'^{[m_0]},I_1^{[0]},\dots, I_q^{[0]};M)\neq 0$. Suppose
that $t>0.$ In this case and without loss of generality we may
assume that $k_1>0$. Since ${\mathcal I}$ is not contained in
$\sqrt{\mbox{Ann}M}$ we have by Proposition \ref{Obs1} that there
exists a weak-(FC)-element $x\in I_1$ with respect to
$(J,J',I_1,\dots,I_q;M)$. By Definition, $x$ is also a
weak-(FC)-element with respect to $(J,I_1,\dots,I_q;M)$ and
$(J',I_1,\dots,I_q;M)$. Since $k_1>0$ and
$e^j(J^{[k_0]},I_1^{[k_1]},\dots, I_q^{[k_q]};M)\neq 0$, $x\in
I_1$ is an (FC)-element with respect to $(J,I_1,\dots,I_q;M)$ (see
Proposition \ref{prop4}). Then by \cite[Proposition
4.4]{perez-bedregal1}
$$e^j(J^{[k_0]},I_1^{[k_1]},\dots, I_q^{[k_q]};M)=e^j(J^{[k_0]},I_1^{[k_1-1]},\dots, I_q^{[k_q]};\overline{M}),$$
where $\overline{M}=M/xM.$   Hence
$$e^j(J^{[k_0]},I_1^{[k_1-1]},\dots, I_q^{[k_q]};\overline{M})\neq 0.$$
\noindent By inductive assumption follows that
$e^j({J'}^{[m_0]},I_1^{[m_1-1]},\dots,
I_q^{[m_q]};\overline{M})\neq 0$ for all $m_1-1\leq
k_1-1,...,m_q\leq k_q, m_0>0$ such that
$j+m_0+(m_1-1)+...+m_q=D-1.$ By \cite[Proposition
4.4]{perez-bedregal1}, we have
$$e^j({J'}^{[m_0]},I_1^{[m_1-1]},\dots, I_q^{[m_q]};\overline{M})=e^j({J'}^{[m_0]},I_1^{[m_1]},\dots, I_q^{[m_q]};M).$$
\noindent Thus, $e^j({J'}^{[m_0]},I_1^{[m_1-]},\dots,
I_q^{[m_q]};M)\neq 0$ for all $m_1\leq k_1,...,m_q\leq k_q,$
$m_0>0$ such that $j+m_0+m_1+...+m_q=D.$ The induction is
complete.

The proof of $(ii):$ By  \cite[Theorem 4.1]{perez-bedregal1} it
follows that
$$\ell\left(\frac{{\bold I}^{\bold r}M_{p+m}}{J^m{\bold I}^{\bold r}M_p}\right)$$ \noindent is a polynomial of degree $D$ for large
$p,m, {\bold r}.$ The terms of total degree $D$ in this polynomial
are
$$P(m,p,{\bold r})=\sum_{k_0+|\bold k|+j=D}\frac{e^j(J^{[k_0]},I_1^{[k_1]},\dots, I_q^{[k_q]};M)}{k_0!j!k_1!\cdots k_q!}
m^{k_0}r_1^{k_1}\cdots r_q^{k_q}p^j.$$ \noindent In particular, if
$r_1=...=r_q=n\gg 0$ we have
$$\ell\left(\frac{{\bold I}^{\bold r}M_{p+m}}{J^m{\bold I}^{\bold r}M_p}\right)=
\ell\left(\frac{I_1^{n}\cdots I_q^{n}M_{p+m}}{J^mI_1^{n}\cdots
I_q^{n}M_p}\right)=\ell\left(\frac{I^nM_{p+m}}{J^mI^nM_p}\right).$$
Hence
$$P(m,p,n,\dots,n)=S(m,n,p)=\sum_{k_0+i+j=D}\frac{e^j(J^{[k_0]},I^{[i]};M)}{k_0!j!i!}
m^{k_0}n^{i}p^j.$$ \noindent From this it follows that
$$\sum_{k_0+i+j=D}\frac{e^j(J^{[k_0]},I^{[i]};M)}{k_0!j!i!}
m^{k_0}n^{i}p^j\!=\!\!\!\!\!\!\!\!\sum_{k_0+|\bold
k|+j=D}\!\!\!\!\!\frac{e^j(J^{[k_0]},I_1^{[k_1]},\dots,
I_q^{[k_q]};M)}{k_0!j!k_1!\cdots k_q!} m^{k_0}n^{|\bold k|}p^j.
$$
\noindent Comparing both terms we get
$$e^j(J^{[D-j-i]},I^{[i]};M)=\sum_{|\bold
k|=i}(i!)\frac{e^j(J^{[D-j-i]},I_1^{[k_1]},\dots,
I_q^{[k_q]};M)}{k_1!\cdots k_q!}$$ \noindent for all $i\leq D-j.$

The proof of $(iii)$: By $(ii)$ we have
$$e^j(J^{[D-j-i]},I^{[i]};M)= \sum_{|\bold
k|=i}(i!)\frac{e^j(J^{[D-j-i]},I_1^{[k_1]},\dots,
I_q^{[k_q]};M)}{k_1!\cdots k_q!}.$$ \noindent From this it follows
that
$$\max\{i| e^j(J^{[D-j-i]},I^{[i]};M)\neq
0\}\!\!=\!\!\max\{i \!\!=\!\!|\bold
k||e^j(J^{[D-j-i]},I_1^{[k_1]},\dots, I_q^{[k_q]};M)\neq 0\}.$$
\noindent Hence from $(i)$ and Theorem \ref{Teo1.2} we get
$$
\begin{array}{lll}
\vspace{0.3cm} l&=&\max\{i| e^j({\mathfrak
mG_1}^{[D-j-i]},I^{[i]};M)\neq
0\}\\
\vspace{0.3cm} &=&\max\{i=|\bold
k||e^j(J^{[D-j-i]},I_1^{[k_1]},\dots, I_q^{[k_q]};M)\neq 0\}.
\end{array}
$$
\noindent and
$$
\max\{i=|\bold k||e^j(J^{[D-j-i]},I^{[i]};M)\neq 0\} =\max
L^*_U(M).
$$
Therefore, $l=\max L^*_U(M).$

Now, assume that $x_1,\dots,x_l$ is a maximal (FC)-sequence in
$I$ with respect to $U.$ By Proposition
\ref{prop4} it follows that
$$\dim(\mbox{Supp}(M/(x_1,\dots,x_l)M:{\mathcal I}^{\infty}))=\dim(\mbox{Supp}(M/(0_M:{\mathcal I}^{\infty}))-l.$$
\noindent In particular we have that $\cal I$ is not contained in
$\sqrt{\mbox{Ann}(M/(x_1,...,x_l)M)},$ and hence, by Proposition
\ref{Obs1}, there exists a weak-(FC)-element $x\in I$ such that
$x_1,\dots,x_l,x$ is a weak-(FC)-sequence in $I$ with respect to
$(J,I;M).$ Let $k$ be the length of
 a maximal weak-(FC)-sequence in $I$ with respect to $(J,I;M).$ By Theorem \ref{teo1} $(ii)$,
$k=s({ I}).$ Hence $l+1\leq k=s({ I}).$

The proof of $(iv):$ We may assume that $\mbox{ht}({\mathcal
I}+\mbox{Ann}(M)/\mbox{Ann}(M))>0$, for otherwise the result is
trivial. By Proposition \ref{prop4} $(vi)$ we have
$$e^j(J^{[k_0]},I_1^{[k_1]},\dots, I_q^{[k_q]};M)\neq 0$$
\noindent  for $|\bold k|\leq \mbox{ht}({\mathcal
I}+\mbox{Ann}(M)/\mbox{Ann}(M))-1.$ Hence by \cite[Theorem
4.6]{perez-bedregal1},
$$\mbox{ht}({\mathcal I}+\mbox{Ann}(M)/\mbox{Ann}(M))-1\leq
\min L^*_U(M).$$ The proof is complete.
\end{proof}

By Theorem \ref{Teo1.2}, $e^j(J^{[k_0]},I_1^{[k_1]},\dots,
I_q^{[k_q]};M)\neq 0$ if and only if there exists an (FC)-sequence
$x_1,...,x_t$ $(t=k_1+\cdots +k_q)$ with respect to
$U=(J,I_1\dots,I_q;G)$ consisting of $k_1$ elements in
$I_1$,...,$k_q$ elements in $I_q$. If $\mbox{ht}(\mathcal
I)=s(I),$ then $\mbox{ht}(\mathcal I)-1=l=s( I)-1.$ As a
consequence of Theorem \ref{teo3} and Propositions \ref{prop5}
$(vii)$ we immediately get the following result.

\begin{cor}\label{cor1} Assume that $\mbox{ht}(\mathcal
I)=h>0$ and let $J$ be finitely generated $R$-submodule of
$G_1$ of finite colength. Then
\begin{itemize}
\item [(i)] $e^j(J^{[k_0]},I_1^{[k_1]},\dots, I_q^{[k_q]};G)= 0$ if $|\bold k|\geq
s( I).$

\item [(ii)] $e^j(J^{[k_0]},I_1^{[k_1]},\dots,
I_q^{[k_q]};G)\neq 0$ if $|\bold k|\leq h-1$ and if $x_1,...,x_t$
$(t=|\bold k|)$ is a weak-(FC)-sequence in $I_1\cup\dots \cup I_q$
with respect to $U$ consisting of $k_1$ elements in
$I_1$,...,$k_q$ elements in $I_q,$ then
$$e^j(J^{[k_0]},I_1^{[k_1]},\dots,
I_q^{[k_q]};G)=e^{j}_{BR}(J,{\overline G}),$$ where
$\overline{G}=G/(x_1,...,x_t)G.$

\item [(iii)] If
$\mbox{ht}(\mathcal I)=s({ I}),$ then
$e^j(J^{[k_0]},I_1^{[k_1]},\dots, I_q^{[k_q]};G)\neq 0$ if and
only if $|\bold k|\leq h-1.$
\end{itemize}
\end{cor}

\section{Aplication for modules}

The most important application of the theory of mixed
multiplicities and (FC)-sequences just developed is in the context
of  families of $R$-submodules of  free modules, as we will
explain below.

 Let $(R, \mathfrak{m})$ be a Noetherian local ring. For any submodule $E$
 of the free $R$-module $R^p,$ the symmetric algebra
$G:=\mbox{Sym}(R^p)=\oplus S_n(R^p)$ of $R^p$ is a polynomial ring
$R[T_1,\ldots ,T_p].$ If $h=(h_1,\ldots ,h_p)\in R^p,$ then we
define the element $w(h)=h_1T_1 + \ldots +h_pT_p\in
S_1(R^p)=:G_1.$ We denote by ${\cal R}(E):=\oplus {\cal R}_n(E)$
the subalgebra of $G$ generated in degree one by $\{w(h): h\in
E\}$ and call it the {\it  Rees algebra} of $E$.  Given any
finitely generated $R$-module $N$ consider the graded $G$-module
$M:=G\otimes_R N.$ We are now ready to recall  the notion of mixed
multiplicities for family of submodules  in $R^p,$ as defined in
\cite{perez-bedregal1}. Here the linear submodules of $G_1$ of the
previous sections will be replaced by a module $E.$

Let $F,E_1,\dots,E_q$ be finitely generated $R$-submodules of
$R^p$ with $F$ of finite colength and denote by $J,I_i$ the
$R$-submodule of $G_1$ generated by ${\cal R}_1(J)$ and ${\cal
R}_1(E_i),$ for all $i=1,\dots, q,$ respectively.  Let ${\mathcal
I}$ be the ideal of $G$ generated by $I_1\cdots I_q$. The
$j^{\mbox{th}}$-{\it mixed multiplicity}
$e^j(F^{[k_0]},E_1^{[k_1]},\dots,E_q^{[k_q]};N)$ of the modules
$F,E_1,\dots, E_q,$ with respect to $N,$ are defined by
$$e^j(F^{[k_0]},E_1^{[k_1]},\dots,E_q^{[k_q]};N)=e^j(J^{[k_0]},I_1^{[k_1]},\dots,I_q^{[k_q]};M)$$
for all $j,k_0,k_1,\dots,k_q\in \N$ with $j+|{\bf k}|=D,$ where
$D=\dim \left (N/(0_N:_N {\cal I}^{\infty})\right )+p-1.$ We call
$e^0(F^{[k_0]},E_1^{[k_1]},\dots,E_q^{[k_q]};N)$ the {\it mixed
multiplicity} of $F,E_1,\dots, E_q,$ with respect to $N,$ and it
is denoted by $e(F^{[k_0]},E_1^{[k_1]},\dots,E_q^{[k_q]};N).$

A sequence of elements $h_1,\ldots, h_q,$ with $h_i\in E_i,$ is a
{\it weak-(FC)-sequence} (resp. {\it (FC)-sequence}) for
$E_1,\ldots, E_q$ with respect to $(E_1,\dots,E_q; N)$ if  the
sequence $w(h_1),\ldots, w(h_k)$ is a weak-(FC)-sequence (resp.
(FC)-sequence)  with respect to $(I_1,...,I_q; M).$

It is now an easy matter to translate into this context all the
results of the previous sections, obtaining in this way important
information on mixed multiplicities and (FC)-sequences for a
family of modules.

\end{document}